\documentclass[letterpaper, 10 pt, conference]{ieeeconf}  

\IEEEoverridecommandlockouts                              

\usepackage{graphicx}
\newtheorem{theorem}{Theorem}
\newtheorem{lemma}{Lemma}
\newtheorem{definition}{Definition}
\newtheorem{assumption}{Assumption}

\newtheorem{property}{Property}
\newtheorem{remark}{Remark}
\newtheorem{corollary}{Corollary}
\usepackage{comment,color}
\usepackage{algorithm}
\usepackage{algorithmic}
\usepackage{boxedminipage}
\usepackage[justification=centering]{caption}
\usepackage{amsmath}
\usepackage{amssymb}
\usepackage[dvipsnames]{xcolor}
\newcommand{\an}[1]{{\color{black}#1}}
\newcommand{\fy}[1]{{\color{black}#1}}
\newcommand{\us}[1]{{\color{black}#1}}

 \newcommand{\remove}[1]{}
\newcommand{\EXP}[1]{\mathsf{E}\!\left[#1\right] }

\def\sF{\mathcal{F}}

\def\Real{\mathbb{R}}
\def\g{\gamma}

\def\e{\epsilon}
\def\a{\alpha}

\title{\LARGE \bf
Stochastic quasi-Newton methods for
 non-strongly convex problems: convergence and rate analysis
}

\author{Farzad Yousefian$^{1}$ and Angelia Nedi\'c$^{2}$ and Uday V. Shanbhag$^{3}$
\thanks{$^{1}$Assistant Professor, School of Industrial Engineering \& Management, Oklahoma State University, Stillwater, OK 74078, USA
        {\tt\small farzad.yousefian@okstate.edu}}%
\thanks{$^{2}$Associate Professor, Industrial \& Enterprise Systems Engineering, University of Illinois at Urbana-Champaign, Urbana, IL 61801, USA
        {\tt\small angelia@illinois.edu}}
\thanks{$^{3}$Associate Professor, Industrial \& Manufacturing Engineering, Pennsylvania State University,  University Park, PA 16802, USA
        {\tt\small udaybag@psu.edu}}%
}

\begin{document}

\maketitle
\thispagestyle{empty}
\pagestyle{empty}

\begin{abstract}
Motivated by applications in optimization and machine learning, we consider
stochastic quasi-Newton (SQN) methods for solving stochastic optimization problems. In
the literature, the convergence analysis of these algorithms relies on strong convexity of the
objective function. To our knowledge, no theoretical analysis is provided for the rate statements in
the absence of this assumption. Motivated by this gap, we allow the objective function to be
merely convex and we develop a cyclic regularized SQN method where the gradient mapping
and the Hessian approximation matrix are both regularized at each iteration and are updated
in a cyclic manner. We show that, under suitable assumptions on the stepsize and regularization parameters, the objective function value converges to the optimal objective function of the original problem in both almost sure and the expected senses. For each case,
a class of feasible sequences that guarantees the convergence is provided. Moreover, the rate
of convergence in terms of the objective function value is derived. Our empirical analysis on a binary classification problem shows that the proposed scheme performs well compared to both classic regularization SQN schemes and stochastic approximation method.
\end{abstract}

\section{Introduction}
 In this paper, we study a stochastic optimization problem of the form:
\begin{equation}\label{eqn:problem}
\min_{x \in \Real^n} f(x):=\EXP{F(x,\xi(\us{\omega}))},
\end{equation}
where $F: \Real^n\times\us{\Real^d}\to\mathbb{R}$ is a function, the random
vector $\xi$ is defined as {$\xi:\Omega
\rightarrow \Real^d$}, \us{$(\Omega,{\cal F}, \mathbb{P})$ denotes the
associated probability space and the expectation \us{$\EXP{F(x,\xi)}$} is
taken with respect to $\mathbb{P}$}. A variety of applications can be cast as the model \eqref{eqn:problem} (see \cite{Spall03,Ermoliev88,borkar00ode,facchinei02finite}).
An important example in machine learning is the support vector machine
(SVM) problem (\cite{bottou-2010,nocedal15,Tib96}). In such problems, a training set containing a large
number of input/output pairs $\{(u_i,v_i)\}_{i=1}^N\in \Real^{m}\times
\Real$ is given where $v_i \in\{-1,1\}$ is the class' index. The goal is
to learn a classifier (e.g., a hyperplane) $h(x,u)$ where $x$ is the
vector of parameters of the function $h$ and $u$ is the input data. To
measure the distance of an observed output $v_i$ from the classifier
function $h$, a real-valued convex loss function $\ell(h;v)$ is defined. The
objective function is considered as the following averaged loss over the
training set \begin{align}\label{eqn:logistic}f(x):=
\frac{1}{N}\sum_{i=1}^N\ell(h(x,u_i);v_i).\end{align} The preceding objective can
be seen as a stochastic optimization model of the form
\eqref{eqn:problem}, where $F(x,\xi):= \ell(h(x,u);v)$ and
$\xi=(u,v)$.

Although, problem
\eqref{eqn:problem} may be seen as a deterministic problem, \an{the
challenges still arise} \us{when standard deterministic schemes are employed.
In particular, when the expectation is over a general measure space
(making computation of $\nabla_{x} \EXP{F(x,\xi)}$ difficult or
 impossible)  or
the distribution $\mathbb{P}$ is unavailable, standard gradient or
Newton-based schemes cannot be directly applied.} \us{This has led
	to significant research on Monte-Carlo sampling techniques. } Monte
	Carlo simulation methods have been used widely in the literature to
	solve stochastic optimization problems. Of these, sample average
	approximation (SAA) methods~\cite{shap03sampling} and stochastic approximation
	(SA) methods~(\cite{robbins51sa,nemirovski_robust_2009})
	(also referred to as stochastic gradient descent methods in the
	 context of optimization) are among popular approaches. It has been discussed that when the sample size is
	large, the computational effort for implementing SAA schemes does
	not scale with the number of samples and these methods become
	inefficient~(\cite{nemirovski_robust_2009,Farzad2}). \us{SA methods, introduced by Robbins and
	Monro~\cite{robbins51sa}, require the construction of a sequence $\{x_k\}$,
	given a randomly generated $x_0 \in \Real^n$:}
\begin{equation}\label{eqn:SA}\tag{SA}   
x_{k+1}:=x_k-\g_k\nabla \us{F}(x_k,\xi_k), \quad \hbox{for }k\geq 0,
\end{equation}
where $\g_k>0$ denotes the stepsize and $\nabla F(x_k,\xi_k)$ denotes
the sampled gradient of the function $f$ with respect to $x$
at $x_k$. Note that the \an{gradient $\nabla F(x_k,\xi_k)$ is assumed to be
an unbiased estimator of the true value of the gradient $\nabla f(x)$ at $x_k$, and 
assumed to be generated by a stochastic oracle.} SA schemes are
	characterized by several disadvantages, \an{including the poorer
		rate of convergence (than their deterministic counterparts) and
		the detrimental impact of conditioning on their performance.
		In deterministic regimes, when second derivatives
				are available,  Newton schemes and their quasi-Newton counterparts} 
				have proved to
be useful alternatives, particularly from the standpoint of displaying
faster rates of convergence~(\cite{Liu:1989:LMB:81100.83726,Nocedal2006NO}).
		

Recently, there has been a growing interest in applying stochastic variants of quasi-Newton (SQN) methods for solving optimization and large scale machine learning problems. In these methods, $x_k$ is given by the following update rule:
\begin{equation}\label{eqn:SQN}\tag{SQN}
x_{k+1}:=x_k-\g_kH_k^{-1}\nabla \us{F}(x_k,\xi_k), \quad \hbox{for }k\geq 0,
\end{equation}
where $H_k \succeq 0$ is an approximation of the Hessian matrix at iteration $k$ that incorporates the curvature information of the
	objective function \us{within} the algorithm. The convergence of this class of algorithms can be derived under a careful choice of the matrix $H_k$ and the stepsize sequence $\g_k$. In particular, boundedness of the eigenvalues of $H_k$ is an important factor in achieving global convergence in convex and nonconvex problems (\cite{Fukushima01,Bottou09}). While in \cite{Schraudolph07} the performance of SQN methods displayed to be favorable in solving high dimensional problems, Mokhtari et al.~\cite{mokh14} developed a regularized BFGS method (RES) \an{by updating the matrix $H_k$
	according to} a modified version of BFGS update rule to assure convergence. To address large scale applications, limited memory variants (L-BFGS) were employed to ascertain scalability in terms of the number of variables (\cite{nocedal15, Mokhtari15}). In a recent extension~\cite{wang14}, a
		stochastic quasi-Newton method is presented for solving nonconvex
	stochastic optimization problems. Also, a variance reduced SQN method 
	\an{with a constant stepsize was developed \cite{Lucchi15} for smooth strongly convex problems characterized by a 
linear convergence rate}. 
	

\noindent {\bf Motivation:} One of the main assumptions in the developed stochastic SQN method (e.g.~\cite{nocedal15, Mokhtari15}) is the strong convexity of the objective function. Specifically, this assumption plays an important role in deriving the rate of convergence of the algorithm. However, in many applications, the objective function is convex, but not strongly convex \an{such as, for example,} the logistic regression function that is given by $\ell(u^Tx,v):=\ln (1+\exp(-u^Txv))$ for $u,x\in\Real^n$ and $v \in \Real$. While lack of strong convexity might lead to a very slow convergence, no theoretical results on the convergence rate in given in the literature of stochastic SQN methods. 
A \an{simple} remedy to address this challenge is to regularize the objective function with the term $\frac{1}{2}\mu\|x\|^2$ and solve the approximate problem of the form 
\begin{equation}\label{eqn:problem-reg}
\min_{x \in \Real^n} f(x)+\frac{\mu}{2}\|x\|^2,
\end{equation}
where $\mu>0$ is the regularization parameter. A trivial drawback of this technique is that the optimal solution to the approximate problem \eqref{eqn:problem-reg} is not an optimal of the original problem \eqref{eqn:problem-reg}. Importantly, choosing $\mu$ to be a small number deteriorates the convergence rate of the algorithm. This issue is resolved in SA schemes through employing averaging techniques for non-strongly convex problems and they display the optimal rate of 
$\mathcal{O}\left(\frac{1}{\sqrt{k}}\right)$ \an{(see~\cite{nemirovski_robust_2009,Nedic2014}).} 
A limitation to averaging SA schemes is that boundedness of the gradient mapping is required to achieve such a rate. 

\noindent {\bf Contributions:} Motivated by these gaps, in this paper, we consider stochastic optimization problems with non-strongly convex objective functions and Lipschitz but possibly unbounded gradient mappings. We develop a so-called cyclic regularized stochastic BFGS algorithm to solve this class of problems. Our framework is general and can be adapted within other variants of SQN methods. Unlike the classic regularization, we allow the regularization parameter $\mu$, denoted by $\mu_k$,  to be updated and decay to zero through implementing the iterations. This enables the generated sequence to approach to the optimal solution of the original problem and also, benefits the scheme by guaranteeing a derived rate of convergence. A challenge in employing this technique is to maintain the secant condition and ascertain the positive definiteness of the BFGS matrix. We overcome this difficulty by carefully updating the regularization parameter and the BFGS matrix in a cyclic manner. We show that, under suitable assumptions on the stepsize and the regularization parameter (referred to as tuning sequences), the objective function value converges to the exact optimal value in an almost sure sense. Moreover, we show that under different settings, the algorithm achieves convergence in mean and we derive and upper bound for the error of the algorithm in terms of the tuning sequences. We complete our analysis by showing that under a specific choice of the tuning sequences, the rate of convergence in terms of the objective function value is of the order $\frac{1}{\sqrt[5]{k}}$. 


The rest of the paper is organized as follows. Section \ref{sec:alg} presents the outline of the proposed algorithm addressing problems with non-strongly convex objectives. In Section \ref{sec:conv}, we prove the convergence of the scheme in both almost sure and expected senses and derive the rate statement. We present the numerical experiments in Section \ref{sec:num}. The paper ends with some concluding remarks in Section \ref{sec:conc}.

 \textbf{Notation:} A vector $x$ is assumed to be
a column vector and $x^T$ denotes its transpose, \an{while} $\|x\|$ {denotes} the Euclidean vector norm, i.e.,
$\|x\|=\sqrt{x^Tx}$.
We write \textit{a.s.} as the abbreviation for ``almost
surely''. 
\an{For a symmetric matrix $B$, we write $\lambda_{min}(B)$ to denote its smallest eigenvalue.}
We use $\EXP{z}$ to denote the expectation of a random variable~$z$. A function $f:X \subset
			\mathbb{R}^n\rightarrow \mathbb{R}$ is said to be strongly convex
				with parameter $\mu>0$, if  $f(y)\geq
				f(x)+\nabla f(y)^T(y-x)+\frac{\mu}{2}\|x-y\|^2,$ for any $x,y \in X$. 
	A mapping $F:X \subset
			\mathbb{R}^n\rightarrow \mathbb{R}$
		is Lipschitz continuous with parameter $L>0$ if for any $x, y
			\in X$, we have $\|F(x)-F(y)\|\leq L\|x-y\|$.

\section{Outline of the algorithm}\label{sec:alg}
\us{We begin by  stating} our general assumptions \an{for problem~\eqref{eqn:problem}.} 
The underlying assumption in this paper is that {the function $f$ is convex and smooth}. 
\begin{assumption}\label{assum:convex}  
\begin{enumerate}
\item [(a)]  \an{The function $F(x,\xi)$ is convex with respect to $x$ for any $\xi \in \Omega$.}  
\item[(b)] $f(x)$ is continuously differentiable \us{with  Lipschitz
	continuous gradients} over $\Real^n$ with parameter $L>0$. 
	\item[(c)] \ The optimal solution set of problem \eqref{eqn:problem} is nonempty. 
\end{enumerate}
\end{assumption}
Next, we state the assumptions on the random variable $\xi$ and the properties of the stochastic estimator of the gradient mapping, i.e. $\nabla F$.  
\begin{assumption}\label{assum:main}
\begin{enumerate}
\item[(a)]  Random variables $\xi_k$ are i.i.d.\ for any $k \geq 0$; 
\item[(b)] The stochastic gradient mapping $\nabla F(x, \xi)$ is an unbiased estimator of $\nabla f(x)$, i.e.
$\EXP{\nabla F(x,\xi)}=\nabla f(x)$, and has bounded variance,
	i.e., there exists a \us{scalar} $\nu>0$ such that $\EXP{\|\nabla F(x,\xi)-\nabla f(x)\|^2} \leq \nu^2$ for any $x \in \mathbb{R}^n$.
\end{enumerate}
\end{assumption} 
To solve \eqref{eqn:problem}, we propose a regularization algorithm that generates a sequence $\{x_k\}$ for any $k \geq
0$:
\begin{align}\label{eqn:cyclic-reg-BFGS}\tag{CR-SQN}
x_{k+1}:=x_k -\gamma_k\left(B_k^{-1}+\delta_k \mathbf{I}\right)\left(\nabla F(x_k,\xi_k)+ \mu_k x_k\right).
\end{align}
Here, $\g_k>0$ denotes the stepsize at iteration $k$, $B_k$ denotes the approximation of the Hessian matrix, $\mu_k>0$ is the regularization parameter of the gradient mapping where 
\begin{align}\label{eqn:mu-k}
  \begin{cases}
   \mu_{k+1}=\mu_{k},      &   \text{if } k \text{ is even}\\
    \mu_{k+1}<\mu_{k},  &   \text{if } k \text{ is odd},
  \end{cases}
\end{align}
\an{while} $\delta_k >0$ is the regularization parameter of the matrix $H_k$. 
\fy{We assume when $k$ is even, $\mu_{k}<\mu_{k-1}$ is chosen such that $\nabla F(x_{k},\xi_k)+\mu_{k}x_k\neq 0$.\footnote{\an{Note that we can ensure this non-zero condition, as follows. If we have 
$\nabla F(x_{k},\xi_k)+\mu_{k}x_k=0$ and $x_k\ne0$, then by replacing $\mu_k$ with a smaller value the relation will hold.
If $x_k=0$, then we can draw a new sample of $\xi_k$ to satisfy the relation. 
We can get stuck at $x_k$ with sampling if $0$ solves the original problem.}} 
}
Let us define the matrix $B_k$ by the following rule:
\begin{align}\label{eqn:B-k}B_{k+1}:=
  \begin{cases}
    B_k-\frac{B_ks_k \an{s_k^T} B_k}{s_k^TB_ks_k}+\frac{{y_k^B}(y_k^B)^T}{s_k^T{y_k^B}} + \rho\mu_k \mathbf{I},      &    k \text{ even}\\
    B_{k},  &    k \text{  odd},
  \end{cases}
\end{align}
where for an even $k$,
\begin{align*}&s_k := x_{k+1}-x_k,\cr 
&{y_k^B}:=  \nabla F(x_{k+1},\xi_k) -\nabla F(x_k,\xi_k)  +  (1-\rho) \mu_ks_k,\end{align*}  
and $
0<\rho < 1$ is the regularization factor of the matrix $B_{k+1}$ at iteration $k$. 
To state the properties of the matrix $B_k$, we \an{start by defining the regularized function.}
\begin{definition}\label{def:regularizedF}
Consider the sequence $\{\mu_k\}$ of positive \an{scalars}.
The regularized function $f_k: \Real^n\to\mathbb{R}$ is defined as follows:
\[f_k(x):= f(x)+\frac{\mu_k}{2}{\|x\|^2},\quad  \hbox{for any } k \geq 0.\]
\end{definition}
Similar notation can be used for the regularized stochastic function 
$F_k$ as $F_k(x,\xi) := F(x,\xi)+\frac{\mu_k}{2}\|x\|^2$. We can now define the term $y^{reg}_k$ as the difference between the value of the regularized stochastic gradient mappings at two consecutive points as follows:
\begin{align*}
y^{reg}_k := \nabla F_k(x_{k+1},\xi_k) -\nabla F_k(x_k,\xi_k). 
\end{align*}
 
In the following result, we show that at iterations that the matrix $B_k$ is updated, the secant condition is satisfied implying that $B_k$ is well-defined. Also, we show that $B_k$ is positive definite for $k \geq 0$.
\begin{lemma}\label{BFGS-matrix}
\an{Let Assumption~\ref{assum:convex}(a) hold, and 
let $B_{k}$ be given by} the update rule \eqref{eqn:B-k}. 
\fy{Suppose $B_0 \succeq \rho \mu_0\mathbf{I}$ is a symmetric matrix. 
Then, for any even $k$, \an{the secant condition holds, i.e., $s_k^T{y_k^B} >0$, and $B_{k+1}s_k={y}^{reg}_k$}. Moreover, for any $k$, $B_k$ is symmetric and $B_{k+1} \succeq \rho  \mu_k \mathbf{I}$.}
\end{lemma}
\begin{proof}
\an{It can be easily seen, by the induction on $k$, that all $B_k$ are symmetric when $B_0$ is symmetric, assuming that the matrices are well defined.}
\an{
We use the induction on even values of $k$ to show that the other statements hold and that the matrices are well defined. 
Suppose $k\geq 2$ is even and for any even values of t with $t < k$, we have $s_t^T{y_t^B} >0$, $B_{t+1}s_t={y}^{reg}_t$, 
and $B_{t+1} \succeq \rho \mu_t \mathbf{I}$. 
We show that all of these relations hold for $t=k$, i.e., 
$s_k^T{y_k^B} >0$, $B_{k+1}s_k={y}^{reg}_k$, 
and $B_{k+1} \succeq \rho \mu_k \mathbf{I}$. 
First, we prove that the secant condition holds. We can write
\begin{align*}
s_k^T{y_k^B} 
&= (x_{k+1}-x_k)^T( \nabla F(x_{k+1},\xi_k) -\nabla F(x_k,\xi_k)  \cr 
&\quad+  (1-\rho )\mu_k(x_{k+1}-x_k)) \cr
&\geq (1-\rho )\mu_k\|x_{k+1}-x_{k}\|,
\end{align*}
where we used the convexity of $F(\cdot,\xi)$.} 
 
\an{From the induction hypothesis, $B_{k-1} \succeq \rho\mu_{k-2}\textbf{I}$ since $k-2$ is even.
Furthermore, since $k-1$ is odd,
we have $B_{k}=B_{k-1}$ by the update rule \eqref{eqn:B-k}.  
Therefore, $B_k$ is positive definite.
Note that since $k$ is even, the choice of $\mu_k$ is such that $\nabla F(x_k,\xi_k)+\mu_kx_k\neq 0$ (see the discussion
following~\eqref{eqn:mu-k}).
Since $B_k$ is positive definite, the matrix $B_k^{-1}+\delta_k\mathbf{I}$ is positive definite.
Therefore, 
we have $\left(B_k^{-1}+\delta_k \mathbf{I}\right)\left(\nabla F(x_k,\xi_k)+ \mu_k x_k\right) \neq 0$, implying that 
$x_{k+1} \neq x_k$. Hence
\[s_k^T{y_k^B} \geq (1-\rho )\mu_k\|x_{k+1}-x_{k}\|^2 >0,\]
where we used $\rho<1$.
Thus, the secant condition holds.}

Also, since $k-1$ is odd, by update rule~\eqref{eqn:mu-k}, it follows that $\mu_k=\mu_{k-1}$. 
From the update rule \eqref{eqn:B-k} and that $B_k$ is positive definite and symmetric, we have
\begin{align}\label{B_k:pos-def}
 B_{k+1}&
 &\succeq B^{0.5}\left(\textbf{I}-\frac{B_k^{0.5}s_k (s_k)^T B_k^{0.5}}{\|B_k^{0.5}s_k\|^2} \right)B^{0.5}
 + \rho\mu_k \mathbf{I},
\end{align}
where \an{the last relation is due to ${y_k^B}(y_k^B)^T\succeq  0$ and $s_k^T{y_k^B>}0$. 

Since $x_{k+1} \neq x_k$ we have $s_k^TB_ks_k \neq 0$. Thus,
the matrix $\textbf{I}-\frac{B_k^{0.5}s_k (s_k)^T B_k^{0.5}}{\|s_k^TB_k^{0.5}\|^2}$ is well defined and 
positive semidefinite, since it is symmetric with eigenvalues are between 0 and 1.}

Next, we show that $B_{k+1}$ satisfies $B_{k+1}s_k=y_k^{reg}$.  
Using the update rule \eqref{eqn:B-k}, \an{for even $k$} we have,
\begin{align*}
  B_{k+1}s_k& 
  = B_ks_k \an{ -\frac{B_ks_k s_k^T B_ks_k}{s_k^TB_ks_k} +\frac{ y_k^B (y_k^B)^T s_k}{s_k^T{y_k^B}} }+ \rho\mu_k s_k \cr 
& = B_ks_k \an{- B_ks_k + y_k^B }+ \rho\mu_k s_k \cr 
& = \nabla F(x_{k+1},\xi_k) -\nabla F(x_k,\xi_k)  \an{+ \mu_k s_k}.
\end{align*}
Since $k$ is even, we have $\mu_{k+1}=\mu_k$ implying that
\begin{align*}
  B_{k+1}s_k   
&  = \nabla F(x_{k+1},\xi_k) -\nabla F(x_k,\xi_k)  \cr &+  \mu_{k+1}x_{k+1}-\mu_kx_k \cr 
& = \nabla F_k(x_{k+1},\xi_k) -\nabla F_k(x_k,\xi_k) = y_k^{reg},
\end{align*}
 where the last equality follows by the definition of the regularized mappings.
 
From the preceding discussion, we conclude that the induction hypothesis holds also for $t=k$. 
Therefore, all the desired results hold for any even $k$. 
To complete the proof, we need to show that for any odd $k$, we have 
$B_{k+1} =B_k \succeq \rho \mu_{k} \textbf{I}$. 
By the update rule~\eqref{eqn:B-k}, we have $B_{k+1}=B_k$. 
Since $k-1$ is even, $B_k \succeq \rho \mu_{k-1} \textbf{I}$. Also, from \eqref{eqn:mu-k} 
we have $\mu_k=\mu_{k-1}$. Therefore,  $B_{k+1} =B_k \succeq \rho \mu_{k} \textbf{I}$. 
\end{proof}

\section{\an{Convergence Analysis}}\label{sec:conv}
In this section, we analyze the convergence properties of the stochastic
recursion ~\eqref{eqn:cyclic-reg-BFGS}. The following assumption provides the
required conditions on the stepsize sequence $\g_k$ and is a commonly
used assumption in the regime of stochastic approximation
methods~\cite{mokh14,wang14,nocedal15}.

\an{
\begin{property}\label{proper:propsfk}[Properties of the regularized function]
The function $f_k$ from Definition~\ref{def:regularizedF} for any $k \geq 0$ has the following properties:
\begin{itemize}
\item [(a)] $f_k$ is strongly convex with a parameter $\mu_k$.
\item [(b)] $f_k$ has Lipschitzian gradients with parameter $L+\mu_k$.
 \item [(c)] $f_k$ has a unique minimizer over $\Real^n$, denoted by $x^*_k$. Moreover, for any $x \in \Real^n$, 
 \[2\mu_k (f_k(x)-f_k(x^*_k)) \leq \|\nabla f_k(x)\|^2,\] 
 \[ \|\nabla f_k(x)\|^2\leq 2(L+\mu_k) (f_k(x)-f_k(x^*_k)).\]
\end{itemize}
\end{property}
}
%
\an{
The existence and uniqueness of $x^*_k$ in Property~\ref{proper:propsfk}(c) 
is due to the strong convexity of the function $f_k$ (see, for example, Sec. 1.3.2 in~\cite{Polyak87}), 
while the relation for the gradient is known to hold for a 
strongly convex function with a parameter $\mu$ that also has Lipschitz gradients with a parameter 
$L$ }\fy{(see Lemma 1 in page 23 in~\cite{Polyak87}). }

The next result provides an important property \an{for} the
recursion~\eqref{eqn:cyclic-reg-BFGS} that will be subsequently used to show the
convergence of the scheme. Throughout, we let $\sF_k$ denote the history of the method up to time $k$, i.e., 
$\sF_k=\{x_0,\xi_0,\xi_1,\ldots,\xi_{k-1}\}$ for $k\ge 1$ and $\sF_0=\{x_0\}$. 
Also, we denote the stochastic error of the \an{regularized gradient} estimator by 
\begin{align}
\label{def:wk}w_k:= \nabla F(x_k,\xi_k)-\nabla f(x_k), \an{\quad\hbox{for all }k\ge0.} 
\end{align}

\begin{lemma}\label{lemma:F-ineq}
[A recursive error bound inequality]
Consider the algorithm \eqref{eqn:cyclic-reg-BFGS}. 
Suppose sequences $\g_k$, $\delta_k$, and $\mu_k$ are chosen such that for any $k \geq 0$, $\mu_k$ satisfies \eqref{eqn:mu-k}, and  
\an{
\begin{align}\label{ineq:BoundCondition}
(L+\mu_k)^2 \g_k \left((\rho\mu_{k-1})^{-1}+\delta_k\right)^2\leq  \delta_k\mu_k.
\end{align}
}
Under Assumptions~\ref{assum:convex} and \ref{assum:main}, for any $k\geq 1$ and \an{any optimal solution $x^*$,}
we have
\begin{align}\label{ineq:cond-recursive-F-k}
& \EXP{f_{k+1}(x_{k+1})\mid \sF_k}-f^* \\ \notag
& \leq (1-\g_k\delta_k \mu_k)\left(f_k(x_k)-f^*\right) +\g_k\delta_k \frac{\mu_k^2}{2}\|x^*\|^2\cr
&+ \frac{(L+\mu_k)}{2}\g_k^2\left((\rho\mu_{k-1})^{-1}+\delta_k\right)^2\nu^2.
\end{align}
\end{lemma}
\vspace{.1in}
\begin{proof}
The Lipschitzian property of $\nabla f_k$ \an{(see Property~\ref{proper:propsfk}(b)) and} 
the recursion \eqref{eqn:cyclic-reg-BFGS} imply that 
\begin{align*}
&f_k(x_{k+1}) \leq f_k(x_k)+\nabla f_k(x_k)^T(x_{k+1}-x_k)\\ &+\frac{ (L+\mu_k)}{2}\|x_{k+1}-x_k\|^2 \cr
&= f_k(x_k)-\g_k\nabla f_k(x_k)^T\left(B_k^{-1}+\delta_k \mathbf{I}\right)(\nabla F(x_k,\xi_k) +\mu_kx_k)\cr &+ \frac{ (L+\mu_k)}{2}\g_k^2\|\left(B_k^{-1}+\delta_k \mathbf{I}\right)(\nabla F(x_k,\xi_k)+\mu_kx_k)\|^2.\end{align*}
\an{From the definition of the stochastic error $w_k$ (see~\eqref{def:wk}) and 
and the definition of the regularized function (see Definition~\ref{def:regularizedF}), 
we have
\[\nabla F(x_k,\xi_k)+\mu_kx_k=\nabla f(x_k) + \mu_kx_k +w_k= \nabla f_k(x_k) +w_k.\]
Hence,}
\begin{align*}
&f_k(x_{k+1}) \leq  f_k(x_k)\cr
& \ -\g_k\nabla f_k(x_k)^T(B_k^{-1}+\delta_k \mathbf{I})(\nabla f_k(x_k)+w_k)\cr 
& \ + \frac{ (L+\mu_k)}{2}\g_k^2\|\left(B_k^{-1}+\delta_k \mathbf{I}\right)( \nabla f_k(x_k)+w_k)\|^2\\
& \leq  f_k(x_k)-\g_k\lambda_{min}\left(B_k^{-1}+\delta_k \mathbf{I}\right)\|\nabla f_k(x_k)\|^2\cr 
& \ \an{-\g_k\nabla f_k(x_k)^T\left(B_k^{-1}+\delta_k \mathbf{I}\right)w_k}\cr 
& \ + \an{\frac {(L+\mu_k)}{2}\g_k^2\|B_k^{-1}+\delta_k \mathbf{I}\|^2 \left(\|\nabla f_k(x_k)+ w_k\|^2\right)}.
\end{align*}
Note that Lemma \ref{BFGS-matrix}(b) implies that 
\[\delta_k\textbf{I}\preceq B_k^{-1}+\delta_k\an{\textbf{I}} \preceq \left((\rho\mu_{k-1})^{-1}+\delta_k\right)\textbf{I}.\]
\an{From the preceding two relations we obtain
\begin{align*}
&f_k(x_{k+1}) 
\leq f_k(x_k)-\g_k\delta_k \|\nabla f_k(x_k)\|^2\cr 
& \ -\g_k\nabla f_k(x_k)^T\left(B_k^{-1}+\delta_k \mathbf{I}\right)w_k\cr 
&+  \frac{(L+\mu_k)}{2}\g_k^2\left((\rho\mu_{k-1})^{-1}+\delta_k\right)^2 \|\nabla f_k(x_k) +w_k\|^2.
\end{align*}
Next, we take the expected value from both sides with respect to $\sF_{k}$. Note that the matrix $B_k$ and $x_k$ are both deterministic parameters if the history $\sF_k$ is known. Note that from Assumption \ref{assum:main}, $\EXP{w_k\mid \sF_k}=0$ and $\EXP{\|w_k\|^2\mid \sF} \leq \nu^2$. Therefore, 
\[\EXP{\|\nabla f_k(x_k) +w_k\|^2\mid \sF_k} \le \|\nabla f_k(x_k) \|^2+\nu^2.\] 
Thus, we obtain
\begin{align*}
&\EXP{f_k(x_{k+1})\mid \sF_k} 
\leq f_k(x_k)-\g_k\delta_k \|\nabla f_k(x_k)\|^2\cr  
& \ +  \frac{(L+\mu_k)}{2}\g_k^2\left((\rho\mu_{k-1})^{-1}+\delta_k\right)^2 \left(\|\nabla f_k(x_k)\|^2 +\nu^2\right).
\end{align*}
Using Property~\ref{proper:propsfk}(c) for the function $f_k$, we have
\begin{align*}
&\EXP{f_k(x_{k+1})\mid \sF_k} 
\leq f_k(x_k)-2\g_k\delta_k \mu_k \left( f_k(x_k) -f_k(x_k^*) \right)\cr  
& \ +  (L+\mu_k)^2\g_k^2\left((\rho\mu_{k-1})^{-1}+\delta_k\right)^2 \left( f_k(x_k) -f_k(x_k^*) \right)\cr
& \ +  \frac{(L+\mu_k)^2}{2}\g_k^2\left((\rho\mu_{k-1})^{-1}+\delta_k\right)^2 \nu^2.
\end{align*}

Due to our assumption on the choice of the sequences $\g_k$, $\delta_k$ and $\mu_k$, we have 
$(L+\mu_k)^2 \g_k^2\left((\rho\mu_{k-1})^{-1}+\delta_k\right)^2\leq  \g_k\delta_k\mu_k$. 
Thus,
\begin{align}\label{eq:oneo}
&\EXP{f_k(x_{k+1})\mid \sF_k}  \leq  f_k(x_k)-\g_k\delta_k \mu_k\left(f_k(x_k)-f_k(x^*_k) \right)\cr 
&+ \frac{(L+\mu_k)^2}{2}\g_k^2\left((\rho\mu_{k-1})^{-1}+\delta_k\right)^2\nu^2.
\end{align}
}
In the last step, \an{we} 
build a recursive inequality for the error term $f_{k}(x_k)-f^*$. Adding and subtracting $f^*$, we obtain
\an{
\begin{align*}
&f_k(x_k)-f_k(x^*_k) 
= \left(f_k(x_k)-f^*  \right) +\left( f^*-f_k(x^*_k) \right)\cr
& = \left(f_k(x_k)-f^*  \right)  +\left( f_k(x^*)- \frac{\mu_k}{2}\|x^*\|^2-f_k(x^*_k) \right),
\end{align*}
where the last equality follows from $f^*=f_k(x^*) - \frac{\mu_k}{2}\|x^*\|^2$.
Since $x^*_k$ is the minimizer of $f_{k}(x^*_k)$, we have $f_k(x^*)-f_{k}(x^*_k) \geq  0$, implying that
\[f_k(x_k)-f_k(x^*_k) \ge f_k(x_k)-f^* - \frac{\mu_k}{2}\|x^*\|^2.\]
By substituting the preceding inequality in~\eqref{eq:oneo}, we obtain}
\begin{align*}
&\EXP{f_k(x_{k+1})\mid \sF_k} 
 \leq  f_k(x_k)-\g_k\delta_k \mu_k\left(f_k(x_k)-f^*  \right)\cr 
& +\g_k\delta_k \an{\frac{\mu_k^2}{2}}\|x^*\|^2
+ \an{\frac{(L+\mu_k)}{2}}\g_k^2\left((\rho\mu_{k-1})^{-1}+\delta_k\right)^2\nu^2.
\end{align*}
\an{
By subtracting $f^*$ from both sides of the preceding inequality, we see that
\begin{align*}
&\EXP{f_k(x_{k+1}) \mid \sF_k} - f^*
 \leq  \left( 1 -\g_k\delta_k \mu_k\right) \left(f_k(x_k)-f^*  \right)\cr 
& +\g_k\delta_k \frac{\mu_k^2}{2}\|x^*\|^2
+ \frac{(L+\mu_k)}{2}\g_k^2\left((\rho\mu_{k-1})^{-1}+\delta_k\right)^2\nu^2.
\end{align*}
Next, we relate the values $f_{k+1}(x_{k+1})$ and $f_{k}(x_{k+1})$. 
From Definition \ref{def:regularizedF} and $\mu_k$ being \an{non-increasing} we can write 
\begin{align*}f_{k+1}(x_{k+1})&= f(x_{k+1})+\frac{\mu_{k+1}}{2}\|x_{k+1}\|^2  \\
& \leq  f(x_{k+1})+\frac{\mu_{k}}{2}\|x_{k+1}\|^2 = f_k(x_{k+1}).\end{align*}
Therefore, the desired inequality \eqref{ineq:cond-recursive-F-k} holds.
}
\end{proof}
We make use of the following result, which can be found
in~\cite{Polyak87} (see Lemma 11 on page 50).

\begin{lemma}\label{lemma:probabilistic_bound_polyak}
Let $\{v_k\}$ be a sequence of nonnegative random variables, where 
$\EXP{v_0} < \infty$, and let $\{\a_k\}$ and $\{\beta_k\}$
be deterministic scalar sequences such that:
\begin{align*}
& \EXP{v_{k+1}|v_0,\ldots, v_k} \leq (1-\alpha_k)v_k+\beta_k
\quad a.s. \ \hbox{for all }k\ge0, \cr
& 0 \leq \alpha_k \leq 1, \quad\ \beta_k \geq 0, \\
&\quad\ \sum_{k=0}^\infty \alpha_k =\infty, 
\quad\ \sum_{k=0}^\infty \beta_k < \infty, 
\quad\ \lim_{k\to\infty}\,\frac{\beta_k}{\alpha_k} = 0. 
\end{align*}
\an{Then, $v_k \rightarrow 0$ almost surely.}
\end{lemma} 
\fy{In order to apply Lemma \ref{lemma:probabilistic_bound_polyak} to the inequality \eqref{ineq:cond-recursive-F-k} and prove the almost sure convergence, we use the following definitions: 
\begin{align}\label{def:lemma3}
& v_k := f_k(x_k)-f^*, \quad \alpha_k := \g_k\delta_k\mu_k, \notag\\ &\beta_k :=  \g_k\delta_k \mu_k^2\|x^*\|^2+ (L+\mu_k)\g_k^2\left((\rho\mu_{k-1})^{-1}+\delta_k\right)^2\nu^2.
\end{align}
\an{To satisfy the conditions of Lemma \ref{lemma:probabilistic_bound_polyak}, we 
identify a set of sufficient conditions on the sequences $\{\gamma_k\}, \{\mu_k\}$, and $\{\delta_k\}$ in 
forthcoming assumption.}
Later in Lemma \ref{lemma:a.s.sequences}, we provide a class of sequences that meet these assumptions.}
\begin{assumption}\label{assum:sequences}
[Sufficient conditions on sequences for a.s. convergence]
Let the sequences $\{\gamma_k\}, \{\mu_k\}$, and $\{\delta_k\}$ be non-negative and satisfy the following conditions:
\begin{itemize}
\item [(a)]  $\lim_{k \to \infty }\frac{\g_k}{\mu_k^3\delta_k} =0$;
\item [(b)] $\delta_k\mu_{k-1} \leq 1$ for $k \geq 1$;
\item [(c)] $\mu_k$ satisfies \eqref{eqn:mu-k} and $\mu_k \to 0$;
\item [(d)] $\sum_{k=0}^\infty \g_k\delta_k\mu_k =\infty$;
\item [(e)] $\sum_{k=0}^\infty \left(\frac{\g_k}{\mu_k}\right)^2 <\infty$;
\item [(f)] $\sum_{k=0}^\infty \g_k\delta_k\mu_k^2 <\infty$;
\item [(g)] $\g_k\delta_k\mu_k \leq 1$ for $k \geq 0$;
\end{itemize}
\end{assumption}

\an{With Assumption~\ref{assum:sequences}, we have the following  result.}
\begin{theorem}\label{thm:a.s}[Almost sure convergence]
Consider the algorithm \eqref{eqn:cyclic-reg-BFGS}. Suppose Assumptions~\ref{assum:convex},
\ref{assum:main} and~\ref{assum:sequences} hold. 
\an{Then,
 $\lim_{k \to \infty }f(x_k)= f^*$ a.s.}
\end{theorem}
\begin{proof} First, note that from Assumption \ref{assum:sequences}(a), there exists $K\geq 1$ such that for any $k\geq K$ we have \fy{$\frac{4\g_k}{\rho^2\mu_k^3\delta_k} \leq \frac{1}{(L+\mu_0)^2}$}. Taking this into account, using Assumption \ref{assum:sequences}(b) and (c) we can write
\begin{align*}
&\g_k\left((\rho\mu_{k-1})^{-1}+\delta_k\right)^2 \leq \g_k\left(2(\rho\mu_{k-1})^{-1}\right)^2 \\
& \leq \frac{4\g_k}{\rho^2\mu_{k-1}^2} \leq   \frac{4\g_k}{\rho^2\mu_{k}^2}=  (\mu_k\delta_k)\frac{4\g_k}{\rho^2\delta_k\mu_{k}^3} \leq  
 \frac{\delta_k\mu_k}{(L+\mu_k)^2},
\end{align*}
implying that \an{condition~\eqref{ineq:BoundCondition} of Lemma~\ref{lemma:F-ineq} holds.
Hence, relation~\eqref{ineq:cond-recursive-F-k} holds for any $k\geq K$.}
Next, we apply Lemma \ref{lemma:probabilistic_bound_polyak} to prove a.s.\ convergence of 
the algorithm~\eqref{eqn:cyclic-reg-BFGS}. 
\fy{Consider the definitions in~\eqref{def:lemma3} for any $k \geq K$.}
The non-negativity of $\alpha_k$ and $\beta_k$ is implied by the definition and that $\g_k$, $\delta_k$, and $\mu_k$ are positive.
From  \eqref{ineq:cond-recursive-F-k}, we have 
\begin{align*}
& \EXP{v_{k+1}\mid \sF_k} \leq (1-\alpha_k)v_k+\beta_k
\quad  \hbox{for all }k\geq K.
\end{align*}
Since $f^* \leq f(x)$ for any arbitrary $x \in \Real^n$, we can write 
\[v_k=f_k(x_k)-f^*= (f(x_k)-f^*)+\frac{\mu_k}{2}\|x_k\|^2  \geq 0.\]
From Assumption \ref{assum:sequences}(g), we obtain $\alpha_k \leq 1$. Also, from Assumption \ref{assum:sequences}(d), we get $\sum_{k=K}^\infty \alpha_k =\infty$. 
\an{Using Assumption~\ref{assum:sequences}(b) and the definition of $\beta_k$ in~\eqref{def:lemma3}, for an arbitrary solution $x^*$, we can write} 
\begin{align*}& \sum_{k=K}^\infty \beta_k \leq \|x^*\|^2\sum_{k=K}^\infty\g_k\delta_k\mu_k^2+ 4(L+\mu_0)\frac{\nu^2}{\rho^2}\sum_{k=0}^\infty \frac{\g_k^2}{\mu_k^2} <\infty,\end{align*}
where the last inequality is deduced by 
\an{Assumptions~\ref{assum:sequences}(e) and~\ref{assum:sequences}(f).} Similarly, we can write
\begin{align*} \lim_{k\to \infty}\frac{\beta_k}{\alpha_k}& \leq \|x^*\|^2\lim_{k\to \infty}{\mu_k}+ 4(L+\mu_0)\frac{\nu^2}{\rho^2}\lim_{k\to \infty} \frac{\g_k^2\mu_k^{-2}}{\g_k\delta_k\mu_k}  \\ &  = \|x^*\|^2\lim_{k\to \infty}{\mu_k}+4(L+\mu_0)\frac{\nu^2}{\rho^2}\lim_{k\to \infty}\frac{\g_k}{\mu_k^3\delta_k}=0, \end{align*}
where the last equation is implied \an{by Assumptions\ref{assum:sequences}(a) and~\ref{assum:sequences}(c).}
Therefore, all the conditions of Lemma \ref{lemma:probabilistic_bound_polyak} hold and we conclude that 
\an{$v_k:=f_k(x_k)-f^*$ converges to 0 a.s.} 
Let us define $v'_k:= f(x_k)-f^*$ and $v''_k := \frac{\mu_k}{2}\|x_k\|^2$, \an{ 
so that $v_k=v'_k+v''_k$. Since $v'_k$ and $v''_k$ are non-negative, and  $v_k \to 0$ a.s., 
 it follows that $v'_k \to 0$ and $v''_k \to 0$ a.s., implying that
$\lim_{k \to \infty }f(x_k)= f^*$ a.s.}
\end{proof}


\begin{lemma}\label{lemma:a.s.sequences}
Let the sequences $\g_k$, $\delta_k$, and $\mu_k$ be given by the following rules:
\begin{align}\label{equ:seq}
 \g_k=\frac{\g_0}{(k+1)^a}, \quad \delta_k=\frac{\delta_0}{(k+1)^b}, \quad \mu_{k}=\frac{\mu_02^c}{\left(k+\kappa\right)^c},
\end{align}
where $\kappa=2$ if $k$ is even and $\kappa=1$ otherwise, $\g_k$, $\delta_0$, $\mu_0$ are positive scalars such that $\delta_0\mu_0 \leq 2^b$ and $\g_0\delta_0\mu_0 \leq1$, and $a$, $b$, and $c$ are positive scalars that satisfy the following conditions:
\begin{align*}& a>3c+b, \ a+b+c \leq 1\\ & \ a-c > 0.5, \ a+2c+b>1.\end{align*}
Then, the sequences $\g_k$, $\delta_k$, and $\mu_k$ satisfy Assumption \ref{assum:sequences}.
\end{lemma}
\begin{proof}
In the following, we show that the presented class of sequences satisfy each of the conditions listed in Assumption \ref{assum:sequences}:\\
(a) Replacing the sequences by their given rules we obtain \begin{align*} \frac{\g_k}{\mu_k^3\delta_k} & =\frac{\g_0}{8^c\mu_0^3\delta_0}(k+1)^{-a+b}(k+\kappa)^{3c} \\ &\leq  \frac{\g_0}{8^c\mu_0^3\delta_0}(k+1)^{-a+b+3c}.\end{align*}
Since $a>b+3c$, the preceding term goes to zero verfying Assumption \ref{assum:sequences}(a).\\
(b) The given rules \eqref{equ:seq} imply that $\delta_k$ and $\mu_k$ are both non-increasing sequences. Therefore, we have $\delta_k\mu_{k-1} \leq \delta_1\mu_{0}$ for any $k\geq 1$. So, to show that Assumption \ref{assum:sequences}(b) holds, it is enought to show that $\delta_1\mu_{0} \leq 1$. From \eqref{equ:seq} we have $\delta_1=\delta_02^{-b}$. Since we assumed that $\delta_0\mu_0 \leq 2^b$, we can conclude that $\delta_1\mu_{0} \leq 1$ implying that Assumption \ref{assum:sequences}(b) holds.\\
(c) Let $k$ be an even number. Thus, $\kappa=2$. From \eqref{equ:seq} we have $\mu_{k}=\mu_{k+1}=\frac{\mu_02^c}{\left(k+2\right)^c}$. Now, let $k$ be an odd number. Again, according to \eqref{equ:seq} can write 
\[\mu_{k+1}=\frac{\mu_02^c}{\left((k+1)+2\right)^c }<\frac{\mu_02^c}{\left(k+1\right)^c}= \frac{\mu_02^c}{\left(k+\kappa\right)^c}=\mu_k.\]
Therefore, $\mu_k$ given by \eqref{equ:seq} satisfies \eqref{eqn:mu-k}. Also, from \eqref{equ:seq} we have $\mu_k \to 0$. Thus, Assumption \ref{assum:sequences}(c) holds.
$\mu_k$ satisfies \eqref{eqn:mu-k} and $\mu_k \to 0$.\\
(d) From \eqref{equ:seq}, we can write $$\sum_{k=0}^\infty \g_k\delta_k\mu_k =\g_0\delta_0\mu_02^c\sum_{k=0}^\infty (k+1)^{-(a+b+c)} = \infty,$$
where the last inequality is due to the assumption that $a+b+c \leq 1$. Therefore, Assumption \ref{assum:sequences}(d) holds.\\
(e) From \eqref{equ:seq}, we have
\begin{align*}&\sum_{k=0}^\infty \left(\frac{\g_k}{\mu_k}\right)^2 =  \frac{\g_0^2}{\mu_0^24^c}\sum_{k=0}^\infty \frac{(k+\kappa)^{2c}}{(k+1)^{2a}} \\
& \leq  \frac{\g_0^2}{\mu_0^24^c}\left(\sum_{k=0}^1 \frac{(k+\kappa)^{2c}}{(k+1)^{2a}}+\sum_{k=2}^\infty \frac{(2k)^{2c}}{k^{2a}}\right)<\infty \\ 
\end{align*}
where the last inequality is due to $a-c>0.5$. Therefore, Assumption \ref{assum:sequences}(e) is verified.\\
(f) Using \eqref{equ:seq}, it follows
\begin{align*}&\sum_{k=0}^\infty \g_k\delta_k\mu_k^2 =  \g_0\delta_0\mu_0^24^c\sum_{k=0}^\infty (k+\kappa)^{-2c}(k+1)^{a+b} \\
&\leq  \g_0\delta_0\mu_0^24^c\sum_{k=0}^\infty (k+1)^{-(2c+a+b)}<\infty,\end{align*}
where the last inequality is due to \an{$a+2c+b>1$}. Therefore, Assumption \ref{assum:sequences}(e) holds.\\
(g) The rules in \eqref{equ:seq} imply that $\g_k$ ,$\delta_k$ and $\mu_k$ are all non-increasing sequences. We also assumed that $\g_0\delta_0\mu_0 \leq 1$. Hence, $\g_k\delta_k\mu_k \leq 1$ for any $k \geq 1$ and Assumption \ref{assum:sequences}(f) holds.
\end{proof}

\begin{remark}
When $a=0.75$, $b=0$, and $c=0.24$, and $\g_0=\delta_0=\mu_0=0.9$, the Assumption \ref{assum:sequences} is satisfied.
\end{remark}
\begin{assumption}\label{assum:sequences-ms-convergence}[Sufficient conditions on sequences for convergence in mean]
Let the sequences $\{\gamma_k\}, \{\mu_k\}$, and $\{\delta_k\}$ be non-negative and satisfy the following conditions:
\begin{itemize}
\item [(a)] $\lim_{k \to \infty }\frac{\g_k}{\mu_k^3\delta_k} =0$;
\item [(b)] $\delta_k\mu_{k-1} \leq 1$ for $k \geq 1$;
\item [(c)] $\mu_k$ satisfies \eqref{eqn:mu-k};
\item [(d)] There exist $0<\alpha<1$ and $K_1 \geq 0$ such that $\frac{\g_{k-1}}{\mu_{k-1}^3\delta_{k-1}}\leq \frac{\g_{k}}{\mu_{k}^3\delta_{k}}(1+\alpha \g_k\delta_k\mu_k)$ for $k \geq K_1$;
\item [(e)] There exist a scalar $\mathcal{B}>0$ and $K_2 \geq 0$ such that $\delta_k \mu_k^4 \leq \mathcal{B}\g_k$ for $k \geq K_2$;
\end{itemize}
\end{assumption}

\begin{theorem}\label{thm:mean}[Convergence in mean]
Consider the algorithm \eqref{eqn:cyclic-reg-BFGS}. 
Suppose Assumptions \ref{assum:convex}, \ref{assum:main} and \ref{assum:sequences-ms-convergence} hold. 
Then, there exists some $K \geq 0$ such that \begin{align}\label{ineq:bound}
 \EXP{f(x_{k+1})}-f^*\leq \theta\frac{\g_k}{\mu_k^3\delta_k} \quad \hbox{for any } k \geq K,
 \end{align}
  where \an{ $f^*$ is the optimal value of problem~\eqref{eqn:problem}},
  $$\theta = \max \bigg\{\frac{\mu_K^3\delta_Ke_{K+1}}{\g_K},\frac{0.5\mathcal{B}\|x^*\|^2+2(L+\mu_0)\frac{\nu^2}{\rho^2}}{1-\alpha}\bigg\},$$
  with $e_{K+1} := \EXP{f_{K+1}(x_{K+1})}-f^*.$
\end{theorem}
\begin{proof}
Similar to the proof of Theorem \ref{thm:a.s}, from Assumption \ref{assum:sequences-ms-convergence}(a), 
there exists $K\geq 1$ such that for any $k\geq K_0$, the condition \eqref{ineq:BoundCondition} holds and, therefore, 
the inequality \eqref{ineq:cond-recursive-F-k} holds. 
Let $K=\max\{K_0,K_1,K_2\}$. 
Taking expectation from both sides of \eqref{ineq:cond-recursive-F-k}, we obtain \an{for any solution $x^*$},
\begin{align*}
& e_{k+1} \leq (1-\g_k\delta_k \mu_k)e_k+\g_k\delta_k \an{\frac{\mu_k^2}{2}}\|x^*\|^2\notag \\ 
&+ \an{\frac{(L+\mu_k)}{2}}\g_k^2\left((\rho\mu_{k-1})^{-1}+\delta_k\right)^2\nu^2,
\end{align*}

where $e_k:= \EXP{f_k(x_k)}-f^*$. Using Assumption \ref{assum:sequences-ms-convergence}(b), (c) and (e), the preceding inequality yields
\begin{align}\label{ineq:cond-recursive-F-k-expected2}
& e_{k+1} \leq (1-\g_k\delta_k \mu_k)e_k+\frac{\mathcal{B}}{2}\|x^*\|^2\frac{\g_k^2}{\mu_k^2}+ 2(L+\mu_0)\frac{\nu^2}{\rho^2}\frac{\g_k^2}{\mu_k^2}\notag \\
& =  (1-\g_k\delta_k \mu_k)e_k+\mathcal{B}_1\frac{\g_k^2}{\mu_k^2}, \qquad \hbox{for } k \geq K,
\end{align}
where $\mathcal{B}_1:= \frac{\mathcal{B}}{2}\|x^*\|^2+2(L+\mu_0)\frac{\nu^2}{\rho^2}$. We use induction to show the desired result. First, we show that \eqref{ineq:bound} holds for $k=K$. We have 
\begin{align*}e_{K+1}& =\EXP{f_{K+1}(x_{K+1})}-f^*\\ &=\left(\frac{\mu_K^3\delta_K(\EXP{f_{K+1}(x_{K+1})}-f^*)}{\g_K}\right)\frac{\g_K}{\mu_K^3\delta_K}
 \leq \theta \frac{\g_K}{\mu_K^3\delta_K},\end{align*}
implying that \eqref{ineq:bound} holds for $k=K$. Now assume that $e_{k} \leq  \theta\frac{\g_{k-1}}{\mu_{k-1}^3\delta_{k-1}}$, for some $k \geq K$. We show that $e_{k+1} \leq \theta\frac{\g_{k}}{\mu_{k}^3\delta_{k}}$.  From the induction hypothesis and \eqref{ineq:cond-recursive-F-k-expected2} we have
\begin{align*}
& e_{k+1} \leq  (1-\g_k\delta_k \mu_k)\ \theta\frac{\g_{k-1}}{\mu_{k-1}^3\delta_{k-1}}+\mathcal{B}_1\frac{\g_k^2}{\mu_k^2},
\end{align*}
Using Assumption \ref{assum:sequences-ms-convergence}(d) we obtain
\begin{align*}
& e_{k+1} \leq  (1-\g_k\delta_k \mu_k)\theta\frac{\g_{k}}{\mu_{k}^3 \ \delta_{k}}(1+\alpha \g_k\delta_k\mu_k)+\mathcal{B}_1\frac{\g_k^2}{\mu_k^2}.
\end{align*}
The definition of $\theta$ and $\mathcal{B}_1$ imply that the term $\theta(1-\alpha) -\mathcal{B}_1$ is non-negative. It follows 
\begin{align*}
 e_{k+1} &\leq  \theta\frac{\g_{k}}{\mu_{k}^3 \delta_{k}}-\theta(1-\alpha)\frac{\g_k^2}{\mu_k^2} +\mathcal{B}_1\frac{\g_k^2}{\mu_k^2}\\  &=\theta\frac{\g_{k}}{\mu_{k}^3 \delta_{k}}-\left(\theta(1-\alpha) -\mathcal{B}_1\right)\frac{\g_k^2}{\mu_k^2} \leq  \theta\frac{\g_{k}}{\mu_{k}^3 \delta_{k}},
\end{align*}
This shows that the induction argument holds true. Also, we have $f_{k+1}(x_{k+1})= f(x_{k+1})+\frac{\mu_{k+1}}{2}\|x_{k+1}\|^2 \geq f(x_{k+1})$.
Therefore, we conclude that \eqref{ineq:bound} holds. 
\end{proof}
\begin{lemma}\label{lemma:mean-sequences}
Let the sequences $\g_k$, $\delta_k$, and $\mu_k$ be given by \eqref{equ:seq},
where $\g_k$, $\delta_0$, $\mu_0$ are positive scalars such that $\delta_0\mu_0 \leq 2^b$, and $a$, $b$, and $c$ are positive scalars that satisfy the following conditions:
\begin{align*}& a>3c+b, \ a+b < 1, \ -a+4c+b\geq 0.\end{align*}
Then, the sequences $\g_k$, $\delta_k$, and $\mu_k$ satisfy Assumption \ref{assum:sequences-ms-convergence}.
\end{lemma}

\begin{proof}
In the following, we verify the conditions of Assumption \ref{assum:sequences-ms-convergence}.\\
\noindent
Conditions (a), (b), and (c): 
This is already shown in parts (a), (b) and (c) of the proof of Lemma \ref{lemma:a.s.sequences} due to $a>3c+b$, $c>0$, and $\delta_0\mu_0 \leq 2^b$.\\
\noindent
(d) It suffices to show there exist $K_1$ and $\alpha \in (0,1)$ such that for any $k \geq K_1$
\begin{align}\label{ineq:partd}\frac{\g_{k-1}}{\g_k}\frac{\mu_{k}^3}{\mu_{k-1}^3}\frac{\delta_{k}}{\delta_{k-1}}-1\leq \alpha \g_k\delta_k\mu_k.\end{align}
From \eqref{equ:seq}, we obtain 
\begin{align*}&\frac{\g_{k-1}}{\g_k}\frac{\mu_{k}^3}{\mu_{k-1}^3}\frac{\delta_{k}}{\delta_{k-1}}-1
\leq \frac{\g_{k-1}}{\g_k}-1=\left(1+\frac{1}{k}\right)^a-1\\
& = 1+\frac{a}{k}+{o}(\frac{1}{k})-1=\mathcal{O}\left(\frac{1}{k}\right),
\end{align*}
where the first inequality is implied due to both $\mu_k$ and $\delta_k$ are non-increasing sequences, and in the second equation we used the Taylor's expansion of $\left(1+\frac{1}{k}\right)^a$. Therefore, since the right hand-side of the relation \eqref{ineq:partd} is of the order $\frac{1}{k^{a+b+c}}$ and that $a+b+c<1$, the preceding inequality shows that such $\alpha$ and $K_1$ exist such that Assumption \ref{assum:sequences-ms-convergence}(d) holds.\\
\noindent
(e) From \eqref{equ:seq}, we have 
\begin{align*} \frac{\delta_k \mu_k^4}{\g_k}& =\delta_0\mu_02^c(k+\kappa)^{-4c}(k+1)^{a-b} \leq\frac{\delta_0\mu_02^c}{(k+1)^{-a+b+4c}}. \end{align*}
Since we assumed $-a+4c+b\geq 0$, there exists $\mathcal{B}>0$ such that Assumption \ref{assum:sequences-ms-convergence}(e) is satisfied.
\end{proof}

\begin{theorem}\label{lemma:rate}[Rate of convergence]
Consider the algorithm \eqref{eqn:cyclic-reg-BFGS}. Suppose Assumptions \ref{assum:convex} and \ref{assum:main} are satisfied. Let the sequences $\g_k$, $\delta_k$, and $\mu_k$ be given by \eqref{equ:seq} with $a=0.8$, $b=0$, and $c=0.2$, and $\delta_0=\mu_0=0.9$ and $\g_0>0$. Then, \begin{itemize}
\item [(a)] \an{$\lim_{k\to\infty} f(x_k)= f^*$  almost surely, where $f^*$ is the optimal value of problem~\eqref{eqn:problem}. }
\item [(b)] We have \[\EXP{f(x_k)}-f^*=\mathcal{O}(\frac{1}{\sqrt[5]{k}}).\] 
\end{itemize}
\end{theorem}
\begin{proof}
(a) The given values of $a$, $b$, $c$ and $\delta_0$ and $\mu_0$ satisfy the conditions of 
Lemma \ref{lemma:a.s.sequences}. Therefore, all conditions of Theorem \ref{thm:a.s} are met, 
\an{and the desired statement follows.}\\
(b) The given values of $a$, $b$, $c$ and $\delta_0$ and $\mu_0$ satisfy the conditions of Lemma~\ref{lemma:mean-sequences}. Therefore, all conditions of Theorem~\ref{thm:mean} are satisfied, \an{so from \eqref{ineq:bound} we obtain}
 \[\EXP{f(x_{k+1})}-f^*\leq \theta\frac{\g_k}{\mu_k^3\delta_k}= \mathcal{O}\left(\frac{k^{-0.8}}{k^{-0.6}}\right)=\mathcal{O}(\frac{1}{\sqrt[5]{k}}).\]
\end{proof}

\begin{remark}[Computational cost] In large scale settings, a natural concern related to the implementation of algorithm \eqref{eqn:cyclic-reg-BFGS} is the computational effort in calculation of $B_k^{-1}$. An efficient technique to calculate the inverse is the Cholesky factorization where the matrix $B_k$ is stored in the form of $L_kD_kL_k^T$ and only the matrices $L_k$ and $D_k$ are updated at each iteration. This calculation can be done in $\mathcal{O}(n^2)$ operations (see \cite{Nocedal2006NO}). In large scale settings, 
the limited memory variant of the proposed algorithm can be considered which is a subject of our future work.

\end{remark}

\section{Numerical experiments}\label{sec:num}
We consider a binary classification problem studied in \cite{Yeh09} where the goal is to classify the credit card clients into credible and non-credible based on their payment records and other information. The data set is from the UCI Machine Learning repository. There are 23 features including education, marital status, history of past payment and the mount of bill statement in the past six months. We employ the logistic regression loss function given by \eqref{eqn:logistic} where
\begin{align}\label{eqn:logistic2}
\ell(u_i^Tx,v_i):=-v_i\ln(c(x,u_i))-(1-v_i)\ln(1-c(x,u_i))
\end{align}
where $c(x,u_i):=(1+\exp(-u_i^Tx))^{-1}$, $v_i \in \{0,1\}$ characterizes the class' type and $u_i \in \Real^{23}$ represents the vector of features. We use 1000 data points to run the simulations. We compare the performance of the proposed algorithm \eqref{eqn:cyclic-reg-BFGS} with that of the regularized stochastic BFGS (RES) algorithm in \cite{mokh14} and also the SA algorithm \eqref{eqn:SA}. To employ RES, since the objective function \eqref{eqn:logistic} is non-strongly convex, we assume the function is regularized as in \eqref{eqn:problem-reg} for some constant $\mu$. Fig. \ref{fig:fig1} and \ref{fig:fig2} compare the performance the three algorithms. Here we assumed that for CR-SQN, $\rho=0.9$, $\mu_0=\delta_k=1$ for any $k$, and that $\g_k$ and $\mu_k$ are given by \eqref{equ:seq} with $a=0.8$, and $c=0.2$. Also, for RES, we set $\mu=1$, $\delta=1$. In both RES and SA schemes, we use $\g_k=\g_0/(k+1)$. It is observed that in both cases, CR-SQN outperforms RES. Comparing Fig. \ref{fig:fig1} with Fig. \ref{fig:fig2}, we also observe that the SA scheme seems very sensitive to the choice of the initial stepsize $\g_0$ which is known as a main drawback of this scheme.

\begin{figure}[htb]
  \centering
  \includegraphics[scale=.30, angle=0]{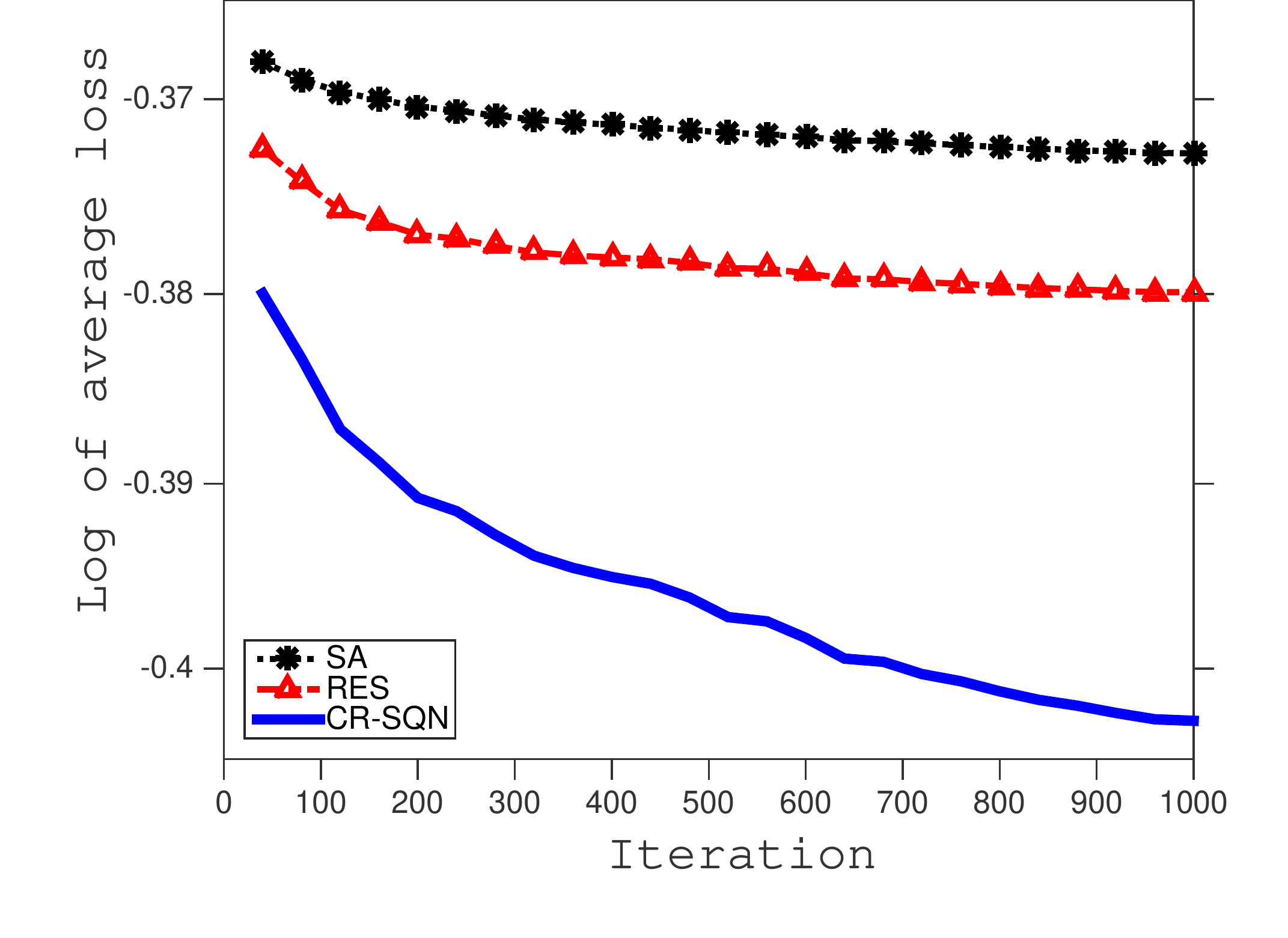}
  \vspace{-.1in}
  \caption{CR-SQN vs. RES vs. SA - $\g_0=0.01$}
  \label{fig:fig1}
\end{figure}
\begin{figure}[htb]
  \centering
  \includegraphics[scale=.30, angle=0]{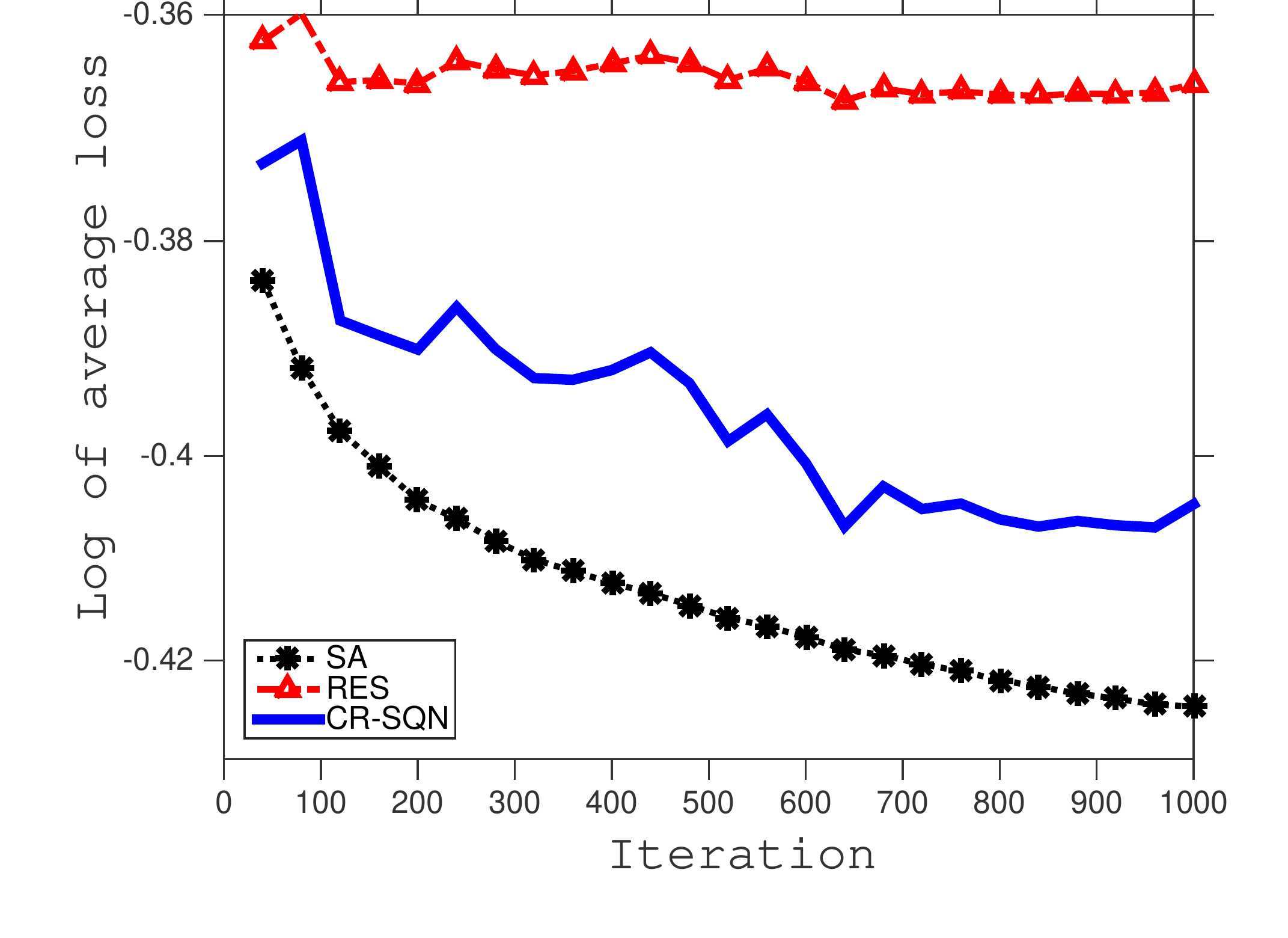}
  \vspace{-.1in}
  \caption{CR-SQN vs. RES vs. SA - $\g_0=0.1$}
  \label{fig:fig2}
\end{figure}
To perform a sensitivity analysis, we compare CR-SQN with RES and SA in Table \ref{tab:table1} and \ref{tab:table2}. In Table \ref{tab:table1}, we report the averaged loss function of CR-SQN and RES for different settings of regularization. We maintain the initial regularization parameter of CR-SQN, $\mu_0$ and the regularization parameter of RES, $\mu$ to be equal. We observe that in all settings, CR-SQN attains a lower averaged loss value. In Table \ref{tab:table2}, we observe that by changing the initial stepsize $\g_0$, except for the case $\g_0=0.1$, CR-SQN outperforms the SA scheme.

\begin{table}[htb] 
\centering 
\begin{tabular}{|c |c|c |c|} 
\hline 
  \multicolumn{2}{|c|}{CR-SQN}&  \multicolumn{2}{c|}{RES} 
\\ 
\hline
$\mu_0$ & ave. loss & $\mu$ & ave. loss \\
\hline

 $ 1$ & $0.6684$ & $1$ & $0.6839$  \\
 \hline
  $ 0.1$ & $0.6683$ & $0.1$ & $0.6949$  \\
   \hline

   $ 0.01$ & $0.8205$ & $0.01$ & $0.9017$  \\
    \hline
    $0.001$ & $4.3090$ & $0.001$ & $5.0424$ \\

\hline
\end{tabular} 
\caption{CR-SQN vs. RES: varying regularization parameter} 
\label{tab:table1} 
\end{table}

\begin{table}[htb] 
\centering 
\begin{tabular}{|c |c|c |c|} 
\hline 
  \multicolumn{2}{|c|}{CR-SQN}&  \multicolumn{2}{c|}{SA} 
\\ 
\hline
$\g_0$ & ave. loss & $\g_0$ & ave. loss \\
\hline

  $ 0.1$ & $0.6673$ & $0.1$ & $0.6540$  \\
   \hline

   $ 0.01$ & $0.6683$ & $0.01$ & $0.6888$  \\
    \hline
    $0.001$ & $0.6901$ & $0.001$ & $0.6927$ \\
\hline
 $ 0.0001$ & $0.6928$ & $0.0001$ & $0.6931$  \\
 \hline
\end{tabular} 
\caption{CR-SQN vs. SA: varying initial stepsize} 
\label{tab:table2} 
\end{table}

\section{Concluding remarks}\label{sec:conc}
To address stochastic optimization problems in the absence of strong convexity, we developed a cyclic regularized stochastic SQN method where at each iteration, the gradient mapping and the Hessian approximate matrix are regularized. To maintain the secant condition and carry out the convergence analysis, we do the \an{regularization} in a cyclic manner. Under specific update rules for stepsize and regularization parameters, our algorithm generates a sequence that converges to an optimal solution of the original problem in both almost sure and expected senses. Importantly, our scheme is characterized by a derived convergence rate in terms of the objective \an{function values.} Our preliminary empirical analysis on a binary classification problem is promising.  
\bibliographystyle{IEEEtran}
\bibliography{reference}
\end{document}